\documentclass{article}
\usepackage[T1]{fontenc}
\usepackage{amsmath,amsthm,amsfonts,amssymb}
\usepackage[mathscr]{eucal}
\usepackage{indentfirst}
\usepackage{graphicx}
\usepackage{exscale,relsize}
\usepackage{multicol}
\usepackage[dvips,final]{epsfig}

\usepackage{xcolor,colortbl}

\theoremstyle{plain}
\newtheorem{theorem}{Theorem}[section]
\newtheorem{corollary}{Corollary}[section]
\newtheorem{proposition}{Proposition}[section]

\newtheorem{lemma}{Lemma}[section]
\newtheorem{definition}{Definition}[section]

\newtheorem{examples}{Examples}[section]

\theoremstyle{definition}

\def\P{{\mathbb P}}
\def\N{{\mathbb N}}
\def\E{{\mathbb E}}

\def\T{{T^{(2)}_n}}

\def\bx{{\bf x}}
\def\by{{\bf y}}
\def\nt{{\tilde n_t}}

\newcommand{\blue}[1]{{\color{blue} #1}}

\title{From the divergence between two measures to the shortest path between two observables }
\author{Miguel Abadi
\thanks{Instituto de Matem\'{a}tica e Estat\'{i}stica, Universidade de S\~{a}o Paulo.}
\and Rodrigo Lambert 
\thanks{ Faculdade de Matem\'atica, Universidade Federal de Uberl\^andia, Uberl\^andia-MG, Brazil.}
\date{}
}

\begin{document}
\maketitle

\begin{abstract}
We consider two independent and stationary measures over $\chi^\N$, where $\chi$ is a finite or countable alphabet. 
For each pair of $n$-strings in the product space
we define $T_n^{(2)}$ as the length of the shortest path connecting one of them to the other.
Here the paths  are generated by the underlying dynamic of the measures.
If they are ergodic and have positive entropy we prove that, for almost every pair of realizations $(\bx,\by)$, $T^{(2)}_n/n$ concentrates in one, as $n$ diverges.
Under mild extra conditions we prove a large deviation principle. 
We also show that the fluctuations of $T_n^{(2)}$ converge (only) in distribution to a non-degenerated  distribution. These results are all linked to a quantity that compute the similarity between those two measures. It is the so-called divergence between two measures, wich is also introduced. 
Along this paper, several examples are provided.

\end{abstract}
\begin{center}
{\bf Running head:} The shortest path between two strings.\\
{\bf Subject class:} 37xx, 41A25, 60Axx , 60C05, 60Fxx. 
\end{center}

\blue{
\tableofcontents
}

\section{Introduction: the shortest-path function}

Suppose one has to built a communication net consisting in nodes and links between nodes.
A question of major interest is how to design the net such that it is easy to communicate from each node to the other
without paying the cost  of constructing a  large number of links.

In this paper we study  a quantity which describes the structural complexity of the net.
Given two nodes, it gives the length of the shortest path from one node to another one.
We consider the case where the nodes are given by the partition in $n$-cylinders or $n$-strings
of the phase space:
Specifically we consider a  finite or countable set $\chi$.
For each $n\in\N$, the nodes correspond to the partition of $n$-cylinders or  $n$-strings of $\chi^\N$.
We consider also two independent probability measures over  $\Omega=\chi^\N$.
The address node is chosen according to a measure $\mu$ and the source node according to a measure $\nu$.
We assume  that both measures are ergodic and that $\mu$ is absolutely continuous with respect to $\nu$,
otherwise the communication could be impossible.
We denote with $\T$ the function that gives the length of the shortest path that communicates two $n$-strings.
The link between this two strings  is driven by the shift operator $\sigma$ 
 over $\Omega$. That is for ${\bf x}=(x_0,x_1,\dots)\in\Omega$ one gets $\sigma{\bf x}=(x_1,x_2,\dots)$. 
 
 Let us introduce the cornerstone for this paper. It is the quantity that gives the minimum number of steps to get from a string to another one, and will be given nextly.

\begin{definition} \label{shortest}
The shortest-path function
is defined by
$$
\T(\bx,\by) = \inf \{k \geq 1 \ : \ y_0^{n-1} \cap \sigma^{-k}(x_0^{n-1}) \neq \emptyset \} .
$$
\end{definition}
Here and ever after we write $x_m^{n}$ as  shorthand of $x_mx_{m+1}\dots x_n$ for any $0 \le m\le n\le \infty$. 
To illustrate this definition, let us take a look in the word ABRACADABRA in three different languages: If $\bx ,\by, \bf z$ are such that $x_0^{10}=$ ABRACADABRA (english), $y_0^{11}=$ AVRAKEHDABRA (aramaic) and $z_0^{12}=$ ABBADAKEDABRA (chaldean).
Then $T_{11}^{(2)}(\bx , \by)=8$ since, considering the firsts 11 letters of $\bx$ and $\by$,
we have to shift 8 times $y_0^{10}$ to be able to connect it with $x_0^{10}$.
Similarly  $T_{11}^{(2)}(\bx , \bf z)=$ $T_{11}^{(2)}(\by , {\bf{z}})=$ $9$. 
Further we have $T_{11}^{(2)}(\by , \bx)= T_{11}^{(2)}(\bf{z} , \by)=$ $T_{11}^{(2)}(\bf{z} , \bx)=$ $10$
and $T_{11}^{(2)}(\bx , \bx) =7, \  T_{12}^{(2)}(\by , \by)=11, \ T_{13}^{(2)}({\bf z} , {\bf z})=12$.

The random variable $T^{(2)}_n$ is a two-dimensional version of the \emph{shortest return function} 
$T_n(\bx)= \inf\{ k \ge 1 \ | \   x_0^{n-1} \cap \sigma^{-k}(x_0^{n-1}) \}$.
That is, it gives the length of the shortest path starting from and arriving to the same node.
Its concentration phenomena has been already studied in \cite{STV, ACS}. A large deviation principle was related to the R\'enyi entropy in \cite{AbVa, HaVa2010, AbCa}. Limiting theorems for  its fluctuations where presented  in \cite{ AbLa, AbGaRa}.
Since $T_n$ considers starting and target sets being the same and $T_n^{(2)}$ allows them coming from different measures, $T_n$ and $\T$  have different nature.
In topological terms: $T_n$ describes a local, while $T_n^{(2)}$ describes a global characteristic of the connection net. 

In this paper we prove three fundamental theorems which describes the net trough the statistical properties of $\T$: Concentration, large deviations and fluctuations.

Firstly we prove that $\T/n$ converges almost surely to one, as $n$ diverges.
Our result holds when $\mu$ has positive entropy and $\nu$ verifies some  specification property
which prevents the net to be extremely sparse.

The concentration of $\T/n$ leads us to  study  its large deviation properties. Namely, the decaying to zero rate of the probability
of this ratio deviating from one. We compute this rate under the additional condition that the measures verify certain regularity condition.
A similar condition was introduced and already related to the existence of a large deviation principle for 
the shortest return function $T_n$ in \cite{AbCa}.

The limiting rate of the large deviation function of $\T$ is determined by a quantity that deserves attention on its own.
It gives a measure of similarity (or difference) between two measures (see definition \ref{div}). In words, it is the expectation of the marginal distribution of order $k$ of one of them with respect to the other. Since it is symmetric, they role are exchangeable in this definition.
We call it the divergence of order $k$. 
We also study  some of its properties that are used later on in the large deviation principle  for $\T$ above mentioned.
We provide several examples.
In many cases the divergence results on an exponentially decreasing sequence on $k$ and this leads to 
consider its limiting rate. 
One of our main results establishes the existence of the limiting rate function which is far from being evident.
We use a kind of sub-additivity property but with a telescopic technique rather the classical linear one.
We show that in particular, when the two measures coincide, this limit corresponds to the R\'enyi entropy of the measure at argument 
$\beta=2$ (see item $(e)$ of examples \eqref{renyi}). 

To describe the complexity of the net, we study the distribution of the shortest path function.
We compute the distribution of a re-scaled version of $\T$ (namely, $n-\T$) and prove that
it converges to a non-degenerated distribution
which depends on the stationary measures $\mu$ and  $\nu$.
The limiting distribution may depend on an infinite number of parameters if the measures do.
This limiting distribution also depends on the divergence between the measures.
As an application of this theorem we compute the proportion of pairs of $n$-strings which do not overlap (wich we call the \emph{avoiding pairs} set).


When the subject is the distribution of $\T$ we are not aware of any work which consider its behaviour in 
the context of stationary measures.
There are some works which consider  models of random graphs and present  empirical data
which adjust the distribution of the shortest path to Weibull or Gamma distributions \cite{BaKeRa, Vaz}. 
But even for classical models, for instance Erd\"os-R\'enyi graph,  its full distribution, in a theoretical sense,  has never been considered in the literature  
\cite{ BGHJ, KNBK, Ukk}.

%
%

Since the random variables $\T$ are defined on the same probability space, we further ask about a stronger convergence.
Our last result shows that \break $n-\T$ does not even converge in probability, and a lower bound for the distance between
two consecutive terms of the sequence $n-\T$ is given.

Finally  we think 
it is important also to highlight the connection  of the shortest path function with the study of the Poincar\'e recurrence statistics.
The waiting time function  introduced by Wyner and Ziv in \cite{WyZi} is a well-studied quantity in the literature.
Given two realizations $\bx, {\bf y} \in \chi^{\N}$, 
it is  the time expected until $x_0^{n-1}$ appears in the realization ${\bf y}$ of another process.  
That is 
$$
W_n({\bf x},{\bf y}) = \inf\{k \geq 1 \ : \ y_k^{k+n-1} = x_0^{n-1}\} \ .
$$
Now, we have that  the shortest path function  is the minimum of  the waiting times of $x_0^{n-1}$, taking the minimum over all the realizations
${\bf z}\in\Omega$ that begin with $y_0^{n-1}$. That is
$$
\T({\bf x},{\bf y}) =  \inf_{ {\bf z} : z_0^{n-1} = y_0^{n-1}}W_n({\bf x},{\bf z}) \  ,
$$
A number of classical results are known for $W_n$.
When both strings are chosen with the same measure, Shields showed
that for stationary ergodic Markov chains $\ln W_n/n \to h$ for almost every pair of realizations, as $n$ diverges and $h$
is the Shannon entropy of the measure \cite{Shi}. 
 Nobel and Wyner \cite{NoWy} had proven a convergence in probability to the same limit. This result holds for $\alpha-$mixing processes with a certain rate function $\alpha$.
Marlon and Shields extended it to weak Bernoulli processes \cite{MaSh}.
Yet, Shields \cite{Shi} constructed an example of a very weak Bernoulli process in which the limit does not hold. 
Finally Wyner \cite{Wyn} proved that the distribution of $W_n({\bf x})\mu({\bf x})$ converges to the exponential law
for $\psi$-mixing measures.
When both strings are chosen with possibly different measures, and the second one is a Markov chain, 
Kontoyiannis  \cite{Kon}
showed that $\ln W_n/n \to h(\mu) + h(\mu||\nu)$ for $h(\mu||\nu)$ the relative entropy of $\mu$ with respect to $\nu$ .

This paper is organized as follows. In section 2 we introduce the divergence concept. Properties, examples and
the proof of its existence are also included in this section. 
In section 3 we prove the concentration phenomena of the shortest path function.
A large deviation principle is proved in section 4.
The convergence of the shortest path distribution is presented in section 5.
An application to calculate the self-avoiding pairs of strings appears also here.
The non convergence in probability is shown in section 6.

\section{The divergence between two measures}

Let $\mu$ and $\nu$ be two probability measures over the same measurable space $(\Omega, \mathcal{F})$. It is natural to ask if this two measures are related in any sense, and if there is some function to scale this relation. There are in the literature several quantities devoted to answer somehow these questions. We highlight here the {\em mutual information} and the {\em relative entropy} (or {\em Kullback-Leibler divergence}), which was been extensively discussed in the literature (see for instance \cite{CoTo}).

The present section is dedicated to a quantity wich describes the degree of similarity between two given measures. As far as we know, it was never considered in the literature. Its definition will be given as follows.

\begin{definition} \label{div}
The $k-$divergence between $\mu$ and $\nu$ is defined by:
$$
\E_{\mu,\nu}(k)
 = \displaystyle\sum_{\omega \in \chi^k}\mu\nu(\omega) \ ,
$$
\end{definition} \noindent (here and ever after, by $\mu\nu(\omega)$ we mean $\mu(\omega)\nu(\omega)$).

Let $\mu_k$ ($\nu_k$) be the projection of $\mu$ $(\nu)$ over the first $k$ coordinates of the space. The $k$-divergence is the mean of  $\mu_k$ with respect to $\nu_k$, or vice-versa. 
It is also the inner product of the $|\chi|^k$-vectors with entries given by the probabilities  $\mu(\omega)$  
and $\nu(\omega), \omega\in\chi^k$  (in any arbitrary ordering of the strings $\omega$).
Notice that $\E_{\mu,\nu}(k) $ is symmetric ($\E_{\mu,\nu}(k) = \E_{\nu,\mu}(k)$), and that it is not null if, and only if, the support of the two measures have non-empty intersection.
The previous sentence can be interpreted as follows: if the two measures do not communicate, the similarity between them is zero.

The next result says that the operation of opening a gap does not produce a smaller result.
As a corollary, we conclude  that 
the $k-$divergence is not increasing in $k$.
For simplicity, hereafter for $\omega\in\chi^n$ and $\xi\in\chi^m$ we denote by $\omega\xi$ the $n+m$-string constructed by concatenation of 
$\omega$ and $\xi$. Formally $\omega \cap \sigma^{-n}\xi$.

\begin{lemma}\label{cauchy_gap}
Let $i+g+j=k$ be non-negative integers.
 Then
\begin{equation}\label{schwartz}
\E_{\mu,\nu}(k)
\leq 
\sum_{\omega \in \chi^i ; \ \zeta \in \chi^j} \mu\nu (\omega \cap  \sigma^{-(i+g)}\zeta) \ . 
\end{equation}
\end{lemma}

\begin{proof}
Consider $\omega \in \chi^i; \  \xi\in\chi^g;   \  \zeta \in \chi^j$.
Let us write the cylinder $ \omega \sigma^{-i}\xi \sigma^{-(i+g)}\zeta \in \mathcal{F}_0^{k-1} $ by concatenating the three cylinders above.
By removing the string $\xi$ in $\mu(\omega\xi\zeta)$
\begin{eqnarray*} 
\sum_{\xi \in \chi^g} \mu\nu\left(\omega\xi\zeta\right) 
& \le &
\sum_{\xi \in \chi^g} \mu\left(\omega \cap \sigma^{-(i+g)}\zeta\right)    \nu\left(\omega\xi\zeta\right) \\
& = &
\mu\nu\left(\omega \cap \sigma^{-(i+g)}\zeta\right) \ .
\end{eqnarray*}
Summing over $\omega$ and $\zeta$ in the last display we get (\ref{schwartz}).
\end{proof}

As a direct consequence of the above proposition, we get that the $k-$divergence is monotonic in $k$.

\begin{corollary}\label{monot}
If $k<l$, then $\E_{\mu,\nu}(k) \geq  \E_{\mu,\nu}(l)$.
\end{corollary}
\begin{proof}
This follows by taking $i=k, g=0, j=l-k$.
\end{proof}
The above corollary proves that the $k-$divergence is a non-increasing function in $k$. In many cases, it decreases at an exponential rate. It is natural to ask about the existence of the limiting rate function, that is
$$
\underline{\mathcal{R}}=   \displaystyle\liminf_{k \to \infty}  \left( -\dfrac{1}{k}\log\E_{\mu,\nu}(k) \right) \ \  ; \ \   
\overline{\mathcal{R}}=   \displaystyle\limsup_{k \to \infty} \left(-\dfrac{1}{k}\log\E_{\mu,\nu}(k) \right) \ .
$$
If both limits are equal, we denote it by $\mathcal{R}$. 
\footnote{Throughout this paper  logarithms can be  taken in  any base.}

In what follows, we provide some examples. They illustrate cases for the existence (or not) of $\mathcal{R}$. Finally we state the main result of this section:  a general condition in which the limiting rate exists. Further, in section \ref{LD}, we will relate this limiting rate with a large deviation principle for $\T$.
\medskip
\begin{examples} 
\item[(a)] Let $\mu$ and $\nu$ two independent measures with disjoint supports. Then $\E_{\mu,\nu}(k) = 0 $ for all $k$, and therefore
$
 \mathcal{R} =  \infty \ .
$
\item[(b)] Suppose that  $\mu$ and $\nu$ concentrate their mass in a unique realization $\bx$ of the process. 
Then, for any  ${\bf y} \in \chi^{\N}$, 
$\mu \nu(y_0^{k-1}) = 1$ if, and only if, $y_0^{k-1} = x_0^{k-1}$ (and  zero otherwise). Thus we get
$
\mathcal{R} = 0 \ .
$
\item[(c)] If both measures $\mu$ and $\nu$  have independent and identically distributed marginals we get

$$
\E_{\mu,\nu}(k) 
= \sum_{\omega \in \chi^k}\mu\nu(\omega) 
= \left[\sum_{x_0 \in \chi}\mu\nu(x_0)\right]^k 
=   \E_{\mu,\nu}(1)^k \ .
$$
Therefore, the limit $\mathcal{R}$ exists and is given by
$$
\mathcal{R}
= - \log E_{\mu,\nu}(1) \ .
$$
\item[(d)] Let $\mu$ be a product of Bernoulli measures with parameter $p$ and $\nu$ a product of Bernoulli measures with parameter $1-p$. Then
$$
\E_{\mu,\nu}(k) = \displaystyle\sum_{x_0^{k-1} \in \chi^k}\prod_{i=0}^{k-1}\mu\nu(x_i) = 2^kp^k(1-p)^k \ .
$$
Then we get
$$
\mathcal{R} =-  \log [2p(1-p)] \ .
$$

\item[(e)] \label{renyi} If  $\mu=\nu$, we get that $\mathcal{R} = H_\mu(2)$, where  
$$
H_\mu(\beta) = - \displaystyle\lim_{k \to \infty}\dfrac{1}{k(\beta -1)}\log \displaystyle\sum_{\omega \in \chi^k}\mu^{\beta}(\omega) \ ,
$$
is the R\'enyi entropy of the measure $\mu$ (provided that it exists).

\item[(f)] A case where the limiting rate does not exist: a sequence that doesn't satisfy the law of large numbers.
Let $\chi = \{0,1\}$, and let $\mu$ be a measure concentrated on the realization:
$$
\bx = 0^{2}1^{2^2}0^{2^3}1^{2^4}\cdots 0^{2^k}1^{2^{k+1}}\cdots
$$
where $a^j$ means the $j$-string $aa\cdots a\in \chi^j$.
On the other hand, let $\nu$ be a product of Bernoulli measures with $p \neq 1/2$. 
By a direct computation, we get that
$
\E_{\mu,\nu}(k) = \nu(x_0^{k-1}) \ .
$
Since the proportion of $0$'s and $1$'s in $x_0^{k-1}$ does not converge as $k$ goes to infinity, we get 
$
\underline{\mathcal{R}} \neq \overline{\mathcal{R}} \ .
$

\item[(g)] Let $\nu$ be an ergodic, positive entropy measure. By the Shannon-McMillan-Breiman Theorem, $- 1/k \log \nu(x_0^{k-1})$ converges to $h(\nu)$, for almost every $\bx \in \chi^{\N}$,where $h(\nu)$ is the entropy of $\nu$. Let $\bx$ be one of such sequences. 
Let $\mu$ be a measure concentrated on $\bx$. Then $\E_{\mu,\nu}(k)= \nu(x_0^{k-1})$ and
$
\mathcal{R} = h(\nu) \ .
$ 
\end{examples}

 The following theorem gives  sufficient conditions for the existence of the limiting rate $\mathcal{R}$. Its proof uses a kind of sub-aditive property. But here, instead of the classical linear iteration of the sub-additivity property, we use a geometric iteration.
To prove the existence of the limiting rate function $\mathcal{R}$ we use a kind of $g$-regular condition
which is a version of the condition introduced in \cite{AbCa}. This condition was used to prove a large deviation principle for the shortest return function $T_n$ of a string to itself.
That principle related the deviations of $T_n$ to the R\'eny entropies of the measure.
 Examples which show its generality and also other properties can be also found there.

In what follows, we present two quantities that will be very usefull troughout this section.
Let $g$ be a fixed non-negative integer. 
For the measure $\mu$, (resp. $\nu$) set
\begin{equation}     \label{psig}
  \psi_{\mu,g}^+(i,j) = \sup_{\omega\in\chi^i, \ \xi\in\chi^j} \dfrac{\mu( \sigma^{-(i+g)}\xi  \ | \ \omega   )}{\mu(\xi)}  \ ,
\end{equation}
and then  
$$
 \psi_{g}^+ =\max\{ \psi_{\mu,g}^+, \psi_{\nu,g}^+ \} \ .
$$

\medskip
Now we are ready to state the main result of the present section. It provides a general condition for the existence of the limiting rate $\mathcal{R}$.

 \begin{theorem} 
  Suppose  there exist positive constants $K> 0$ and $\epsilon>0 $ 
  such that 
 \begin{equation}\label{sublog}
  \log \psi_g^+(i,j) \leq K \dfrac{i+j}{[ \log(i+j)]^{1+ \epsilon}} \ .
  \end{equation}
  Then $\mathcal{R}$ does exist.
  \end{theorem}

  For instance if $\mu$ and $\nu$ have independent marginals,
  we take $g=0$. 
 Immediate calculations give $\psi_{\mu,g}^+(i,j) = \psi_{\nu,g}^+(i,j) = 1 $ for all $i$ and $j$, and condition \eqref{sublog} is satisfied. Moreover, if $\mu$ and $\nu$ are stationary measures of  irreducible, aperiodic,  positive recurrent Markov Chain in a finite alphabet $\chi$ then, 
the Markov property gives
  $$
   \psi_{\mu,0}^+(i,j) =  \sup_{\omega\in\chi  ,  \xi\in\chi} \dfrac{\mu( \sigma^{-1}\xi  \ | \ \omega   )}{\mu(\xi)}   \ ,
  $$
  (resp. $\nu$) which is finite since $\chi$ is finite
   and  \eqref{sublog} is  verified.
   Abadi and Carde\~no \cite{AbCa} constructed several examples  of processes of renewal type, with $g$ equal zero and one
   with exponential or sub-exponential measure of cylinders which verifies \eqref{sublog}. 
   Measures $\mu$ which verify the classical $\psi$-mixing condition, for each $g$ fixed, have $\psi^+_{\mu,g}$ 
   constant and thus \eqref{sublog} holds.
   \\ 
  \\
  Now we present the proof for the Theorem.
  \begin{proof}
Let us take $ \omega\in\chi^i,  \xi \in \chi^g, \zeta \in \chi^j$.
As in the proof of Lemma \ref{cauchy_gap},
$
\sum_{\xi\in\chi^g}  \mu\nu( \omega\xi\zeta) 
\le\mu\nu(\omega  \cap \sigma^{-(i+g)}\zeta) \ .
$
Further, by  \eqref{psig}
$$
\mu(\omega \cap \sigma^{-(i+g)}\zeta) \le
 \psi_{g}^+(i,j)\mu(\omega)\mu(\zeta) \ .
  $$ 
  And the same holds for $\nu$.
 Call $f(k) = \log \E_{\mu,\nu}(k)$. 
 Also, call $c_g(i,j)=  2\log\psi_{g}^+(i,j)$.
 Summing up in $\omega$ and $\zeta$ and taking logarithm, 
 by the inequalities above, we conclude that for all $i, j,$
\begin{eqnarray}
f(i+g+j) 
       & \leq &  \log(\psi_{g}^+(i,j)\psi_{g}^+(i,j))
       + \log\displaystyle\sum_{\omega \in \chi^i} \mu\nu(\omega) 
       + \log\displaystyle\sum_{\zeta \in \chi^j}\mu\nu(\zeta)  \nonumber \\
       & = &   c_g(i,j) +  f(i) + f(j) \ . \label{sub}
\end{eqnarray}
Now we use a kind of sub-additivity argument.
Let $(n_t)_{t\in\N}$ an increasing sequence of non-negative integers such that
\begin{equation} 
\displaystyle\liminf_{n \to \infty}\frac{f(n)}{n} = \displaystyle\lim_{t \to \infty}\frac{f(n_t)}{n_t} \ .
\end{equation}
Consider the sequence $\nt= n_t+g,$ with  $t\in \N$.
Fix $t$. 
Firstly, for any positive integer $n \ge \nt$  write $n=\nt m+r,$ with  positive integers $m, r$ such that $0\le r < \nt$.  
Apply (\ref{sub}) with $i=\nt m-g, j=r$, and gap $g$ to get
\begin{equation} \label{resto}
f(n) \le c_g(\nt m-g, r) +   f(\nt m-g) + f(r) .
\end{equation}
Now, we write $m$ in base 2. 
For this,  there exist a positive integer $\ell(m)$ and non-negative integers $ \ell_1 < \ell_2<...< \ell_{\ell(m)},$
such that    $m= \sum_{s=1}^{{\ell(m)}}  2^{\ell_s}  $.  
Iterating (\ref{sub}) with $i=\nt  \sum_{s=1}^{u-1} 2^{\ell_{s} } - g $ and $j= \nt  2^{\ell_{u}} -g$, 
for $u=2,\dots,{\ell(m)}$, we have that for the middle term in the right hand side of \eqref{resto}
\begin{equation} \label{eme}
f(\nt m-g) \le  
     \sum_{u=2}^{{\ell(m)}}  \ c_g \left(  \nt   \sum_{s=1}^{u-1} 2^{\ell_{s}} -g, \nt 2^{\ell_{u}}-g \right) 
+   \sum_{u=1}^{{\ell(m)} }  f(\nt 2^{\ell_{u}}-g) .
\end{equation}
The first sum in the righthand side is zero in case $\ell(m)=1$.
Finally, we decompose the argument in the last summation.
For any $n\in\N$ of the form $n=\nt 2^\ell-g$ we apply (\ref{sub}) with $i=j=\nt 2^{\ell-1}-g$, 
to get
$$
f(\nt 2^\ell-g) \le c_g(\nt 2^{\ell-1}-g, \nt 2^{\ell-1}-g) +  2 f(\nt 2^{\ell-1}-g) .
$$
An iteration of the above inequality leads to
\begin{equation} \label{potencia}
f(\nt 2^\ell-g) \le \sum_{s=0}^{\ell-1} 2^{\ell-s-1} c_g( \nt 2^{s}-g, \nt 2^{s}-g) +    2^{\ell} f(\nt-g) .
\end{equation} 
Observe that $\nt-g=n_t$.
Collecting \eqref{resto}, \eqref{eme}, \eqref{potencia}
we conclude that the limit superior of $f(n)/n$
is upper bounded by the limit superior of $I+II+III+IV$ where
\begin{eqnarray*}
I&=& \frac{ c_g(\nt m-g, r) }{\nt m-g+r}  \ , \\
II&=&  \frac{1}{\nt m}
           \sum_{u=2}^{{\ell(m)}}  \ c_g \left(  \nt   \sum_{s=1}^{u-1} 2^{\ell_{s}} - g, \nt 2^{\ell_{u}}-g \right) \ , \\
III&=& \frac{1}{\nt m}
          \sum_{u=1}^{{\ell(m)} }
          \sum_{s=0}^{\ell_u-1} 2^{\ell_u-s-1} c_g( \nt 2^{s}-g, \nt 2^{s}-g)  \ ,   \\
IV&=& \frac{ f(r) }{\nt m} + \frac{\sum_{u=1}^{{\ell(m)} } 2^{\ell_u} f(n_t)}{ \nt m }  \ .
\end{eqnarray*}
As $m$ diverges, the first term in $IV$ vanishes since  $0 \le r < \nt$. 
The second one is bounded by $f(n_t)/n_t$. 
$I$ goes to zero by \eqref{sublog}.
We recall that  $II=0$ in case $\ell(m)=1$. Otherwise
we also use condition \eqref{sublog} to get the following upper bound
$$
\dfrac{K}{m} \sum_{u=2}^{\ell(m)} \dfrac{\sum_{s=1}^{u} 2^{\ell_s}}{\left [\log\left ( \nt\sum_{s=1}^{u} 2^{\ell_s}-2g \right )\right]^{1+\epsilon}} \ .
$$ 
The  inner summation  is trivially bounded by
 $3 \leq \sum_{s=1}^{u}2^{\ell_s} \leq 2^{\ell_u +1}$. 
 Since $m=  \sum_{u=1}^{\ell(m)} 2^{\ell_u} $,
 it follows that $II \leq 2K / [\log n_t ]^{1+\epsilon}$. 
Lastly using also \eqref{sublog} we get that
$III$ is upper bounded by
$$
\frac{K}{m}
          \sum_{u=1}^{{\ell(m)} } 2^{\ell_u}
\sum_{s=0}^{\ell_u-1} \frac{1}{[\log \nt 2^{s+1}-2g )]^{1+\epsilon}} \ .
$$
Changing the constant $K$ we can take here logarithm  base 2.
The argument in the logarithm is lower bounded by $n_t 2^{s+1}$. 
Now we use that the sum of a decreasing sequence is bounded above by its first term plus the integral definite by the first and last terms in the sum. Thus, 
the rightmost sum in the above display is bounded by
$$
\sum_{s=1}^{\infty} \frac{1}{[ s + \log n_t ]^{1+\epsilon}} 
\leq
\dfrac{1}{[1+\log{n_t}]^{1+\epsilon}} + 
\dfrac{ 1  }{\epsilon[1+\log{n_t}]^{\epsilon}} \ ,
$$
which  goes to zero as $t$ diverges.
 Summarizing we conclude that
\begin{eqnarray*}
\limsup_{n \to \infty}\dfrac{f(n)}{n}
\le \dfrac{f(n_t )}{n_t}   +  \frac{K }{ [\log n_t ]^{1+\epsilon} } + \dfrac{1}{\epsilon[1+\log{n_t}]^{  \epsilon}}
\ .
\end{eqnarray*}
The inequality holds for every $t$. If we take $t \to \infty$ then
we finish the proof since  $f(n_t)/n_t$ converges by hypothesis.

  \end{proof}

\section{Concentration}

The present section is dedicated to study asymptotics for $\T$. 
Intuition says that the bigger is  $n$, the more difficult is to connect two $n$-strings. 
Thus we expect $\T$  increasing.
The question is at which  rate.
The main result of this section says that  $T_n/n$ converges almost surely to one. 
The proof is divided in two parts. First part proves that the limit inferior is lower-bounded by one and the second one states that the limit superior is upper-bounded by one. For the last one we assume that the process verifies the very weak specification property (see def. \eqref{spec} below).

There are several definition of specification in the literature of dynamical systems.
The first one was introduced by Bowen\cite{Bow2}. Many others appeared, mainly following him with some divergence in the nomenclature
\cite{Sig,KH}, or in weaker forms (see for instance \cite{SuVaYa,Var1,Ya}). 
Basically they mean that, for  any  given set of strings, they can be observed (at least in one single realization of the process) with bounded gaps between them. Sometimes it is required  the realization to be periodic.
for simplicity to the reader, we present our condition here. It is easy to see that it is verified  for  a large class of stochastic processes  and is less restrictive than the previous ones. Examples are provided below.

\begin{definition}\label{spec}
 $(\chi^{\N}, \mu, \sigma)$  
 is said to have the very weak specification property (VWSP)
 if 
 there exists a function  
 $g:\N\to\N$ with $\lim_{n\to\infty}g(n)/n=0,$ that verifies the following:
 For any pair of strings $\omega, \xi  \in \chi^{n}$,  
  there exists a  $\bx \in \chi^\mathbb{N}$ such that, 
$$ x_0^{n-1}=\omega
\quad \textstyle{and} \quad
x_{n+g(n)}^{2n+g(n)-1}=\xi  \ .
$$
\end{definition}

\begin{examples}
\item[(a)] Any process with complete grammar  verifies definition \ref{spec} with $g(n)=0$. We recall that a probability measure $\mu$ defined over $\chi^{\N}$ is said to have complete grammar  if, for all  $n\in\N$ we get
$\mu(\omega) >0 $ for all  $\omega \in \chi^n $.

\item[(b)]  An irreducible and aperiodic  Markov chain over a finite alphabet $\chi$ and stationary measure $\mu$ verifies the VWSP
with $g(n)\le |\chi|$.


\item[(c)]  
We first construct a renewal process $(X_n)_{n \in \mathbb{N}}$  as an image of  the House of Cards  Markov chain
$(Y_n)_{n \in \mathbb{N}}$ with
irreducible and aperiodic transition matrix $Q$ given by
\begin{eqnarray*}
Q(y,0)&=& 1-q_y , \\
Q(y,y+1)&=& q_y , 
\end{eqnarray*}
$y\in \{0,1,2,\dots\}$.
Figure 1 represents the transitions of this process.
\begin{figure}[h]
\centering
\includegraphics[width=0.6\columnwidth]{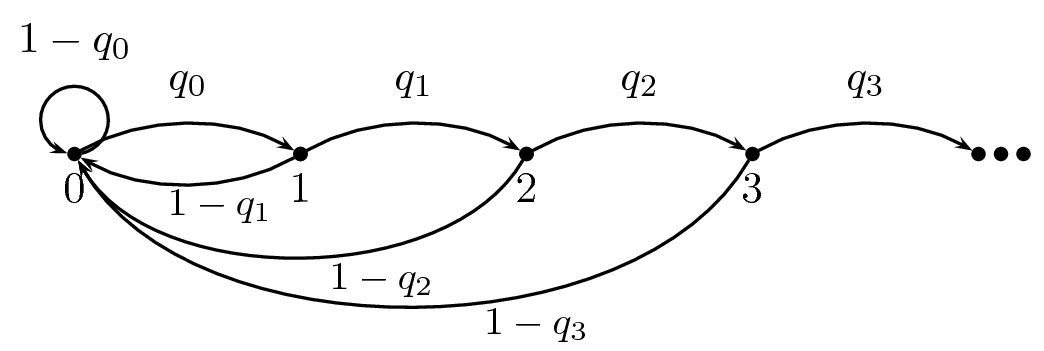}
\caption{House of cards Markov chain $(Y_n)_{n\geq 0}$} \label{ejemploq}
\end{figure}


Let $X_n=0$ if $Y_n\not=0$, and $X_n=1$ if $Y_n=0$, indicating the "renewal" of  $(Y_n)_{n\geq 0}$.
Take $q_y= 1$, for  all $n^2 \le y \le n^2+n$ for some $n\in \N$, and any other $0<q_y<1$ for the remaining coefficients 
to warrant that the Markov chain is positive recurrent.
 Obviously $(X_n)_{n\geq 0}$ has not complete grammar.
 It is easy to see that $g(n)\le \sqrt{n}$ and that this bound is actually sharp, taking
$\omega=\xi=01^{n^2-1} \in \chi^{n^2}$. Thus $(X_n)_{n \geq 0}$ verifies the VWSP.
 On the other hand, the stationary measure of the House of Cards Markov chain itself is an example that does not verify the VWSP.


\end{examples}

Now we can state the concentration theorem, which is the main result of this section.

\begin{theorem}\label{linear}
If  $\mu$ has positive entropy, then
\begin{itemize}
\item[(a)]
$\displaystyle\liminf_{n \to \infty} \dfrac{T_n^{(2)}}{n} \geq 1,  \qquad   \mu\times\nu-\mbox{a.e.} $
\item[(b)]
In addition, if  $\nu$ verifies the VWSP, then
$$ \lim_{n \to \infty} \dfrac{T_n^{(2)}}{n}  = 1,   \qquad  \mu\times\nu-\mbox{a.e.} $$
\end{itemize}
\end{theorem}

Before proving the above result, let us introduce a family of sets and  a result that will be useful  for the proof. 

\begin{definition}\label{R2}
For each $k \in \{1,\cdots , n-1\}$ define the set of pairs $(\bx, \by)\in \chi^\mathbb{N}\times\chi^\mathbb{N}$ 
such that the firsts $k$ symbols of $x_0^{n-1}$ coincide exactly with the last $k$ symbols of $y_0^{n-1}$.
Namely
$$
R_n^{(2)}(k) = \{ (\bx, \by) \in \chi^{\N}\times\chi^{\N} \ : \ y_{n-k}^{n-1}=x_0^{k-1} \} \ .
$$
\end{definition} 
For instance, if $\bx, \by$ are such that $y_0^{5}=011100$ and $x_0^5=111100$, then $(\bx,\by) \in R_{4}^{(2)}(3)$. \\

The following result establishes a connection between the shortest-path function and the $R^{(2)}_n(i)$- sets.

\begin{lemma}\label{T_nR_n} For $k < n$,  it holds
$$
\{T^{(2)}_n \leq k \} \subseteq \bigcup_{i=n-k}^{n-1}R^{(2)}_n(i) \ .  
$$ In addition, if $\nu$ satisfies the complete grammar condition, then the equality holds.
\end{lemma}
\begin{proof} 
By definition, $(\bx,\by)$ belongs to $\{\T \leq k\}$, if and only if,  there is ${\bf z}\in \chi^\N$ and  $1 \leq i \leq k$ such that 
$z_0^{n-1}=y_0^{n-1}$ and $z_i^{i+n-1}=x_0^{n-1}$.  
In particular, since $n-1\ge i$, we have  $y_i^{n-1}=x_0^{n-i-1}$, which in turns says that $(\bx,\by)\in R_n^{(2)}(n-i)$. For the equality, notice that for any pair $(\bx,\by) \in R^{(2)}_n(i)$, $i \in \{n-k, \cdots, n-1\}$ we get that $x_0^{i-1}=y_{n-i}^{n-1}$. The complete grammar condition assures that there exists ${\bf z}\in \chi^\N$ such that 
$z_0^{n-1}=y_0^{n-1}$ and $z_i^{i+n-1}=x_0^{n-1}$.
This concludes the proof.
\end{proof}

The next lemma gives the key connection with  the divergence of $\mu$ and $\nu$. 

From now on we mean by $\P$ the product measure $\mu \times \nu$.
\begin{lemma}\label{medidaRn} For $1 \le k < n$, and $\nu$ a stationary measure, it holds
$$
 \P( R^{(2)}_n(k) )  =  \E_{\mu,\nu}(k)  \  . 
$$
\end{lemma}
\begin{proof} If $1 \le k < n,$ then  
$$
\P( R^{(2)}_n(k) ) = \P(x_0^{k-1}=y_{n-k}^{n-1}) = \sum_{\omega \in \chi^k} \mu \nu (\omega) \ .  
$$
Since the last term above is equal to $\E_{\mu,\nu}(k)$, we finish the proof.
\end{proof}

Now we are able to prove theorem \ref{linear}.

\begin{proof}[Proof of theorem \ref{linear}]

	 
For item $(a)$, let $h>0$ be the entropy of $\mu$. 
Since $\mu$ is ergodic, the Shannon-Mcmillan-Breiman Theorem says that
$$
-\displaystyle\lim_{n \to \infty}\dfrac{1}{n}\log\mu(x_0^{n-1}) = h \ ,
$$
along $\mu$-almost every   $\bx \in \chi^{\N}$.
By Egorov's Theorem, for every $0< \epsilon < h$,  there exists a  subset $\Omega_{\epsilon}$ of $ \Omega$, where this convergence is  uniform and $\mu(\Omega_{\epsilon}) \geq 1- \epsilon$. 
That is, 
for all $\epsilon > 0$, there exists a $k_0(\epsilon)$ such that for all $k>k_0(\epsilon)$
\begin{equation}\label{sh}
e^{-k(h+\epsilon)} < \mu(x_0^{k-1}) < e^{-k(h-\epsilon)} \ ,
\end{equation}
for all $\bx\in \Omega_\epsilon$.
Making the product with $\nu,$
%
%
and using lemmas \ref{T_nR_n} and  \ref{medidaRn}    we get
\begin{eqnarray*}
\mathbb{P}\left(\{T_n^{(2)} \leq (1-\epsilon) n\} \ \cap  \ \Omega_{\epsilon} \times \Omega \right)
                                                & \leq &     \displaystyle\sum_{j=\lfloor\epsilon n \rfloor}^{n-1}\mathbb{P} \left(R_n^{(2)}(j) \ \cap  \ \Omega_{\epsilon} \times \Omega \right) \\
                            & = &  \displaystyle\sum_{j=\lfloor \epsilon n \rfloor}^{n-1}\displaystyle\sum_{\omega\in\chi^j \cap \Omega_{\epsilon} } \mu\nu(\omega) \\
                            & \leq &  \displaystyle\sum_{j= \lfloor \epsilon n \rfloor}^{n-1}e^{-j(h-\epsilon)} \ ,
\end{eqnarray*}
where the last inequality 
 was obtained using (\ref{sh}). A direct computation gives
\begin{eqnarray*}
\displaystyle\sum_{n=1}^{\infty}\mathbb{P}(T_n^{(2)} \leq (1-\epsilon)  n) 
& \leq &  
  \dfrac{   1          }{    1-e^{-(h-\epsilon)}  }   \ .
\end{eqnarray*} 
By Borel-Catelli's Lemma,  $\{T_n^{(2)} \leq (1-\epsilon) n\}$ occurs only finitely many times. 
We conclude
\begin{equation}\label{fim}
\displaystyle\liminf_{n \to \infty}\dfrac{T_n^{(2)}}{n} \geq 1 - \epsilon \ , \ \ \ \mu\times\nu -  \mbox{a.s. in} \ \ \Omega_{\epsilon}\times \Omega \ .
\end{equation}
Since $\epsilon$ is arbitrary, 
this finishes the proof of $(a)$.

Now we prove item $(b)$. Since definition \ref{spec}  implies that 
	 $\T \leq n + g(n)$, we divide both sides by $n$ and get
	 $$
	 \displaystyle\limsup_{n \to \infty} \dfrac{T_n^{(2)}}{n} \leq 1,  \qquad   \mu\times\nu-\mbox{a.e.}
	 $$
	 Combining this with $(a)$, we finish the proof of item $(b)$.

\end{proof}

\section{Large deviations}\label{LD}

In the previous section we showed that $\T/n$ concentrates its mass in $1$, as $n$ diverges. Here, we present the deviation rate for this limit. Since the VWSP implies that 
$$
\P\left(\dfrac{\T}{n} > 1+ \epsilon\right) = 0, \ \ \forall n > n_0(\epsilon) \ ,
$$
it is  only meaningful to consider the lower deviation.

\begin{definition}
We define the  $\liminf$ and $\limsup$  for  the lower deviation rate, respectively, as
$$
\underline{\Delta}(\epsilon) =  \liminf_{{n\rightarrow\infty}} \dfrac{1}{n}\left|\log\P\left(\dfrac{\T}{n} < 1- \epsilon\right) \right| \ ,
$$
  and  
  $$
 \overline{\Delta}(\epsilon) =   {\limsup_{n\rightarrow\infty}} \dfrac{1}{n}\left| \log \P\left(\dfrac{\T}{n} < 1- \epsilon\right) \right| \  .
 $$
If $\underline{\Delta}=\overline{\Delta}$ we write simply $\Delta$.
\end{definition}

We recall that the complete grammar condition 
assures that \break
 $\T \leq n$.

\begin{theorem} \label{LD}
Let $\mu$ and $\nu$ be two stationary probability measures defined over $\chi^{\N}$. 
Then
\begin{itemize} 
\item[(a)]     
$\underline{\Delta}(\epsilon) \ge \epsilon \underline{\mathcal{R}} \  $ and
$\overline{\Delta}(\epsilon) \le \epsilon \overline{\mathcal{R}} \ . $
\item[(b)]     Suppose that $\nu$ has complete grammar. Then the equalities hold in $(a)$.
\end{itemize}
\end{theorem}

The $\psi_g -$regularity of the measure assures the existence of $\mathcal{R}$.

\begin{corollary} 
Under conditions of theorem \ref{LD},
suppose yet that $\nu$ has complete grammar. Then 
$\Delta(\epsilon) = \epsilon \mathcal{R}$.
\end{corollary}

\begin{proof}[Proof of theorem \ref{LD}] By  lemma \ref{T_nR_n}, we have
$$
\left\{\dfrac{T_n^{(2)}}{n} < 1-\epsilon\right\} \subseteq \displaystyle\bigcup_{j=\lceil n\epsilon\rceil}^{n-1}R_n^{(2)}(j) \ ,
$$
with equality if the process has complete grammar.
In this case, considering just the first set in the union, we also have
$$
R_n^{(2)}(\lceil n\epsilon\rceil) \subseteq \
\left\{\dfrac{T_n^{(2)}}{n} < 1-\epsilon\right\} \ .
$$ 
Thus, by lemma \ref{medidaRn}
$$
\E_{\mu,\nu}(\lceil n\epsilon\rceil)\leq \P\left(\dfrac{T_n^{(2)}}{n} < 1-\epsilon \right) 
\leq \sum_{j=\lceil n \epsilon\rceil}^{n-1}\E_{\mu,\nu}(j)\ .
$$
Now we take logarithm, divide by $n$, take limit and use that the divergence is non increasing.
An exchange of variables ends the proof.
\end{proof}

\section{Convergence in law}
In this section we prove the convergence of the normalized distribution of $\T$ to a non-degenerate distribution. 
We also present several examples and provide an application for the main result of the section.

To state the result we need first to introduce the coefficients that appears in the theorem.  
\begin{definition}
Set $a^{(2)}_{n-1,n}= 0$ and for every $1 \leq k \leq n-2$, define:
\begin{itemize}
\item $a^{(2)}_{k,n}= \displaystyle\sum_{m=k+1}^{n-1} \ \displaystyle\sum_{\omega \not\in \cup_{j=k}^{m-1}R_m(j)}\mu\nu(\omega)$ \ .
\item $a^{(2)}_{k}
= \sum_{m=k+1}^{\infty} \ \displaystyle\sum_{\omega \not\in \cup_{j=k}^{m-1}R^{}_m(j)}\mu\nu(\omega)  \ .$
%
\end{itemize}

Here, we define by $R_n(k)$ a set which is a one dimensional version of $R^{(2)}_n(k)$. Namely
$$
R_n(k) = \{ x_0^{n-1} \in \chi^{n} \ :  
 \  x_{n-k}^{n-1}=x_0^{k-1} \} \ .
$$

\end{definition}


Now we can state the main result of this section.

\begin{theorem}\label{distrib} 
Suppose $\nu$ has complete grammar.
 Then, for all $1 \leq k \leq n-1$, it holds:
\begin{itemize}
\item[(a)] $\mathbb{P}(n-T^{(2)}_n \geq k) = \E_{\mu,\nu}(k) + a^{(2)}_{k,n}$ \ .
\item[(b)] $\displaystyle\lim_{n\to \infty}\P(n-\T \geq k ) = \E_{\mu,\nu}(k) + a^{(2)}_k$ \ .
\end{itemize}
\end{theorem}

In the next examples we discuss several cases  of applications of the above theorem.

\begin{examples}
\item[(a)]  The theorem does not warrant that the limiting object $(\E_{\mu,\nu}(k) + a^{(2)}_k)_{n\in\N}$ actually defines a distribution law.
For instance, take $\mu$  concentrated on the unique sequence $\bx=(11111\dots)$.
Let $\nu_p$ be a product of Bernoulli measures with success probability $p$. Let $0<\lambda<1,$ and define
$\nu=\lambda \mu + (1-\lambda) \nu_p$.
Clearly, $\mu$ is absolutely continuous with respect to $\nu$, which has complete grammar.
It is easy to compute  $\P(\T=1)=\nu(1^n)=\lambda+(1-\lambda)p^n$.
Thus $\P(n-\T =\infty) \ge \lambda$ and $n-\T$ does not converge to a limiting distribution.

\item[(b)] Under mild conditions one gets that the limiting object is actually a distribution.
For that, it is enough to give conditions in which
$\P(n-\T = \infty ) = \lim_{k\to \infty} \E_{\mu,\nu}(k) + a^{(2)}_k =0$. 
Directly from its definition  $a^{(2)}_k  \le  \sum_{j=k+1}^{\infty}\E_{\mu,\nu}(j) $.
Thus, the limiting function defines a distribution if $\sum_{j=k}^{\infty}\E_{\mu,\nu}(j) $ goes to zero as $k$ diverges.
It holds if $\min\{\max_{\omega\in \chi^n}\mu(\omega), \max_{\omega\in \chi^n}\nu(\omega)\}$ is summable.
Notice that this is not the case in the example above. 
\item[(c)] When $\mu=\nu$, we recover in the limit, the same limit  distribution of 
the re-scaled  shortest return function $n-T_n$.
In particular, if $\mu$ is a product measure, we recover the limit distribution obtained in \cite{AbLa}.

\item{(d)}
The following example shows a process that has complete grammar,  
and then $n-\T$ converges. 
On the contrary, the re-scaled shortest  return function $n-T_n$ does not converge as shown in \cite{AbGaRa}.
This is due to the fact that the process  is not $\beta-$mixing.
The process $(X_n)_{n\geq 0}$  is defined over $\chi=\{0,1\}$ in the following way.
Let $X_0$ be uniformly chosen over $\{0,1\}$ and independent of everything.
The remaining variables are conditionally independent given $X_0$ and defined by 
$$\mu(X_{2n}=X_0)=1-\epsilon=1-\mu(X_{2n-1}=X_0) .$$
It is obvious that  the process has complete grammar,  has a unique invariant measure
with marginal distribution of $X_n, n\ge 1$ being the uniform one.
Take now $\nu=\mu$. So, they verify the hypothesis of theorem \ref{distrib} and therefore $n-\T$ converges.
\end{examples}







\subsection*{Application}


We call    $\{ (x_0^{n-1}, y_0^{n-1}) \in \chi^n\times\chi^n \ | \ \T(\bx,\by)=n \}$
the set of \emph{avoiding pairs}, since only in this case the chosen two strings do not overlap.
As far as we know it was never considered in the literature.
A similar quantity was actually considered,  the set of \emph{self-avoiding strings}. It is defined similarly,
but using the shortest return function $T_n$, 
instead of the shortest path $\T$.
 Namely $\{T_n = n\}$.
It is read as the  set of strings which do not overlap itself. 
It was first studied in \cite{RoTo}, when the authors treated a problem related to Cellular Automata. 
They considered only the case of a uniform product measure. 
Using an argument due to S. Janson, the authors calculate the proportion of self-avoiding strings of length $n$.
This result was generalized in \cite{AbLa}, and posteriorly  in \cite{AbGaRa}
to independent and $\beta$-mixing processes respectively.

The next result follows immediately from theorem \ref{distrib}, and gives us  the probability of the set of avoiding pairs.

\begin{corollary} Under the conditions of theorem \ref{distrib},
the measure of the avoiding pairs set is given by
\begin{itemize}
\item[(a)] $\P(\T = n ) = 1-\E_{\mu,\nu}(1)-a^{(2)}_{1,n}$.
\item[(b)] $\displaystyle\lim_{n \to \infty}\P(\T = n ) = 1-\E_{\mu,\nu}(1)-a^{(2)}_{1}$.
\end{itemize}
\end{corollary}

\vskip.8cm
Now we proceed to the proof of the main result of this section.

\begin{proof}[Proof of Theorem \ref{distrib}] The main idea of this proof is to bring the two-dimentional problem to a  one-dimensional one.
By  lemma \ref{T_nR_n} 
and since  $\nu$ has complete grammar, we get that
\begin{equation*}
\left\{  n-T_n^{(2)} \geq k \right\}  =  \ \displaystyle\bigcup_{j=k}^{n-1}R_n^{(2)}(j) \ .
\end{equation*}
Decompose the right-hand side of the above equality in disjoint sets to get
\begin{equation} \label{ident}
R^{(2)}_n(k) \  \cup \ \displaystyle\bigcup_{j=k+1}^{n-1} \     R^{(2)}_n(j) \setminus \cup_{l=k}^{m-1}R^{(2)}_n(l)   \ .
\end{equation}
By lemma \ref{medidaRn}, $\mathbb{P}(R^{(2)}_n(k)) = \E_{\mu,\nu}(k)$.
For the second set in (\ref{ident}), 
 $(\bx, \by) \in R^{(2)}_n(j)$, if and only if
$x_0^{j-1} =y_{n-j}^{n-1}$.
Further, $(\bx, \by)$ does not belong to  $\displaystyle\cup_{l=k}^{m-1}R^{(2)}_n(l)$ if, and only if 
$x_0^{l-1} \neq y_{n-l}^{n-1}$, for all $k \le l \le m-1$.
Since this last two conditions depend only on the values of $x_0,\dots, x_{m-1}, y_{n-m},\dots, y_{n-1}$ we get that the probability of the 
rightmost set  in (\ref{ident}) is equal to 
\begin{eqnarray}\label{cres}
 \displaystyle\sum_{j=k+1}^{n-1}
 \displaystyle\sum_{\omega \not\in \cup_{l=k}^{m-1}R_m(l)}\mu\nu(\omega) 
  =  a_{k,n}^{(2)} \ .
\end{eqnarray}
Since $\E_{\mu,\nu}(k)$ does not depend on $n$, the limit of the probability of the left-side  set in   \eqref{ident} when $n$ goes to infinity only depends on its second term, which is a non-decreasing function on $n$. 
Each term is also bounded above by $1$.
Therefore it converges, and this concludes the proof.
\end{proof}

\section{Non-convergence in probability}

In the present section we show that the convergence of $n-\T$ cannot be stronger than  convergence in distribution. The result is stated as follows.
\begin{proposition}
Under the conditions of  theorem \ref{distrib},
 $n-\T$ does not converge in probability.
\end{proposition}

\begin{proof}

It is sufficient to show that 
$$
\P(|n+1 - T^{(2)}_{n+1} - (n-\T)| > \epsilon) ,
$$
does not converge to zero.
Take $0 < \epsilon < 1$. It is obvious that
$$
\{T^{(2)}_{n+1}=\T \} \subset \{|n+1 - T^{(2)}_{n+1} - (n-\T)| > \epsilon\} \ .
$$
Conditioning on $\T=k$
\[
\P(T^{(2)}_{n+1}=\T) 
 =   \sum_{k=1}^{\infty} \P(T^{(2)}_{n+1}=k \ | \ \T=k ) \P(\T=k)  \ .
\]
Since $\nu$ has complete grammar, the above sum goes just up to $n$.
Further, to get $T^{(2)}_{n+1}=k$ whenever  one has $T^{(2)}_{n}=k$, it is necessary and sufficient 
to have $y_n=x_{n-k}$, due also to the complete grammar.
Thus, $\P(T^{(2)}_{n+1}=k \ | \ \T=k )= \sum_{y_n\in\chi}\mu\nu(y_n)=\E_{\mu,\nu}(1)$ which is positive since $\mu$ is absolutely continuous respect to $\nu$.
This finishes the proof.
\end{proof}


	

\section*{Aknowledgements}
We kindly thank A. Rada, B. Saussol and S. Vaienti for useful discussions. We also thank the anonymous referee for her (his) comments and corrections.
This paper is part of R.L.'s Ph.D. Thesis, developed under agreement between University of S\~ao Paulo (CAPES and CNPq SWE-236825/2012-7 grants) and UTLN-France (with CNRS, and BREUDS FP7-PEOPLE-2012-IRSES318999 grants). 
This article is part of the activities of the Project "Statistics of extreme events and dynamics of recurrence" FAPESP Process 	
2014/19805-1.
This article was produced as part of the activities of FAPESP Center for Neuromathematics (grant$\#2013/ 07699-0$ , S.Paulo Research Foundation).

%


\begin{thebibliography}{00}



\bibitem{AbCa} 
 M.\ Abadi and L. \ Cardeno,
{\em Renyi entropies and large deviations for the first-match
function,}
IEEE Trans. Inf. Theory(61), 4 (2015), 1629--1639.





\bibitem{AbLa}
 M.\ Abadi and R.\ Lambert,
{\em The distribution of the short-return function,}
Nolinearity (26) 5 (2013), 1143-1162.

\bibitem{AbGaRa}
 M.\ Abadi, S.\ Gallo and E.\ Rada,
{\em The shortest possible return time of $\beta$-mixing processes.}
Preprint.




\bibitem{AbVa}
M. Abadi and S. Vaienti,
{\em Large Deviations for Short Recurrence,}
Disc. Cont. Dyn. Syst. 21 (2008), 729-747.

%
 

\bibitem{ACS}
V. Afraimovich, J.-R. Chazottes and B. Saussol,
{\em Point-wise dimensions for Poincar\'{e} recurrence
associated with maps and special flows,}
Disc. Cont. Dyn. Syst. 9 (2003), no. 2, 263-280.
(2003).





\bibitem{BaKeRa}
 C. Bauckhage, K. Kersting and B. Rastegarpanah,
{\em  The Weibull as a Model of Shortest Path Distributions in Random Networks,}
Avaliable online in
$http://snap.stanford.edu/mlg2013/submissions/mlg2013\_submission\_5.pdf$



\bibitem{BGHJ} 
V. Blondel, J. Guillaume, J. Hendrickxand and R. Jungers,
{\em Distance distribution in random graphs and application to network exploration,}
Phys. Rev. E 76 (2007), 066101. 





\bibitem{Bow2}
 R. Bowen,
{\em  Periodic points and measures for Axiom A diffeomorphisms,}
Trans. Am. Math. Soc.  154 (1971) 377-397.











\bibitem{CoTo}
T. Cover and J. Thomas,
{\em Elements of Information Theory (Second Edition),} New York: Wiley (1991) 748 pp.











%
%
%
%


 
 
\bibitem{HaVa2010}
 N.\ Haydn and S. \ Vaienti,
 {\em The Renyi entropy function and the large deviation of short return times,}
   Ergodic Theory Dynam. Systems 30 (2010), no. 1, 159 - 179.















\bibitem{KH}
A. Katok and B. Hasselblatt, 
{\em Introduction to the modern theory of dynamical systems,}
Encyclopedia of Math. and its Applications, 54,
Cambridge Univ. Press 1995.



\bibitem{KNBK}
  E. Katzav, M. Nitzan, D. ben-Avraham, P.  Krapivsky, R. Kühn5, N. Ross and O. Biham,
{\em Analytical  results  for  the  distribution  of  shortest
path lengths in random networks.}
EPL (Europhysics Letters) (2015), Volume 111, Number 2, 26006 .

\bibitem{Kon}
I. Kontoyiannis,
{\em Asymptotic Recurrence and Waiting Times for Stationary Processes,}
 J. Theoret. Probab. 11 (1998), no. 3, 795-811.
 







\bibitem{MaSh}
K. Marton and P. Shields, {\em Almost-sure waiting time results for weak and very
weak Bernoulli processes,} Ergod. Th. Dynam. Syst. 15 (1995), 951-960.


\bibitem{NoWy}
A. Nobel and A. D. Wyner,  \emph{A recurrence theorem for dependent processes with applications to data compression,} IEEE Trans Inform. Th. 38 (1992), 1561-1564.

\bibitem{OW}
D.  S. Ornstein and B. Weiss.
{\em Entropy and data compression schemes.}
 IEEE Trans. Inform. Theory 39 (1993), no. 1, 78-83.




%
%


\bibitem{RoTo}
A. Rocha. {\em Substitution Operators.} Phd Thesis, Universidade Federal de Pernambuco (2009).
Avaliable online in {\it{http://toomandre.com/alunos/doutorado/andrea/tese-andrea.pdf}}






\bibitem{Shi} P. Shields,
{\em Waiting Times: Positive and Negative Results on the Wyner-Ziv Problem,}
J. Theor. Probab. (1993) 6, 499-519.


\bibitem{STV} B. Saussol, S. Troubetzkoy and S. Vaienti,
{\em Recurrence, dimensions and Lyapunov exponents,}
J. Stat. Phys. 106 (2002), 623--634.





\bibitem{Sig} K. Sigmund
{\em On dynamical systems with the specification property}
Trans. Amer. Math. Soc. 190 (1974), 285-299.


\bibitem{SuVaYa}
 N. Sumi, P. Varandas and K. Yamamoto,
{\em Partial hyperbolicity and specification.}
Proc. Amer. Math. Soc. 144 (2016), 1161-1170.


\bibitem{Ukk}
A. Ukkonen,
{\em  Indirect estimation of shortest path distributions
with small-world experiments.}
LNCS 8819 - 13th International Symposium, IDA 2014
Leuven, Belgium, October 30 - November 1, 2014
Proceedings 333-344.


\bibitem{Var1}
P. Varandas,
{\em Non-uniform Specification and Large Deviations for Weak Gibbs Measures.}
J. Stat. Phys. 146 (2012), no. 2, 330-358.


\bibitem{Vaz}
A. Vazquez, 
{\em Polynomial Growth in Branching
Processes with Diverging Reproduction Number.
Physical Review Letters, 96(3):038702, 2006.}







\bibitem{Ya}
 K. Yamamoto,
{\em On the weaker forms of the specification property and their applications.}
Proc. Amer. Math. Soc. 137 (2009), no. 11, 3807-3814.






\bibitem{Wyn}
A. Wyner
{\em More on recurrence and waiting times.}
Ann. Appl. Probab. 9 (1999), no. 3, 780-796.

\bibitem{WyZi}
A. \ Wyner and  J.  Ziv.
{\em Some asymptotic properties of the entropy of a stationary
ergodic data source with applications to data compression,}
 IEEE Trans. Inform. Theory 35 (1989), no. 6, 1250-1258.




 \end{thebibliography}
\end{document}